\documentclass[reqno]{amsart}
\usepackage{verbatim}
\usepackage{tikz}
\usepackage{tikz-3dplot}
\usepackage{hyperref}

\theoremstyle{plain}
\newtheorem{theorem}{Theorem}

\newtheorem{lemma}{Lemma}  
\newtheorem{cor}{Corollary} 

\newtheorem{definition}{Definition}
\newtheorem{prop}[lemma]{Proposition}

\theoremstyle{remark}
\newtheorem*{remark}{Remark}

\theoremstyle{definition}

\begin{document}

\title[]{Periodic Orbits in the Kepler-Heisenberg Problem}

\author{Corey Shanbrom}

\address{Department of Mathematics and Statistics \\ California State University, Sacramento \\ Sacramento CA 95819}

\email{corey.shanbrom@csus.edu}

\date{23 November, 2013}



\begin{abstract}  
One can formulate the classical Kepler problem on the Heisenberg group, the simplest sub-Riemannian manifold. 
We take the sub-Riemannian Hamiltonian as our kinetic energy, and our potential is the fundamental solution to the Heisenberg sub-Laplacian. The resulting dynamical system is known to contain a fundamental integrable subsystem. Here we use variational methods to prove that the Kepler-Heisenberg system admits periodic orbits with $k$-fold rotational symmetry for any odd integer $k\geq 3$.  Approximations are shown for $k=3$.
 \end{abstract}

\maketitle


\section{Introduction}
In \cite{MS} we introduced the Kepler-Heisenberg problem and recorded many surprising properties.  
The object of interest is a dynamical system which is intended to model the motion of a planet around a sun if the ambient geometry were the three dimensional Heisenberg group equipped with its sub-Riemannian structure.  

In Hamiltonian mechanics, one typically begins with a Riemannian manifold and a choice of potential energy function.  The Riemannian metric induces a kinetic energy function on the cotangent bundle of the manifold.  On the Heisenberg group, we have a natural choice of kinetic energy, induced by the sub-Riemannian metric, which indeed generates the sub-Riemannian geodesics (see \cite{Tour}).  We choose as our potential the fundamental solution to the Heisenberg sub-Laplacian, given explicitly by Folland in \cite{Folland}.  The delta function source, acting as our sun, lies at the origin.  This characterization of gravitational potential is not original, and is guided by the fact that $\frac{1}{4\pi r}$ is the fundamental solution to the Laplacian on $\mathbb R^3$ (see \cite{Albouy}).

Newton studied the Euclidean Kepler Problem in the 17th century and derived Kepler's three laws of planetary motion.  But the problem was posed on spaces of constant curvature much later.  In 1835, Lobachevsky (\cite{Lobachevsky}) posed the Kepler Problem in three-dimensional hyperbolic space.  Bolyai did similar work (independently) in the same time period.  Paul Joseph Serret posed and solved the Kepler Problem on the two-sphere in 1860.  Schering, Lipschitz, Killing, and Liebmann studied the Kepler Problem on hyperbolic and spherical three-space between 1870 and 1902.  For more information, and the relevant references, see Florin Diacu's wonderful paper \cite{Diacu1}.  With this historical background in mind, it seems natural to continue efforts to pose and solve the Kepler Problem in more general geometries (sub-Riemannian geometry encompasses the Riemannian sort.)

We proved in \cite{MS} that phase space for the Kepler-Heisenberg problem contains a fundamental invariant hypersurface on which the dynamics are integrable. In addition, we reduced the integration of this integrable subsystem to the parametrization of a family of algebraic plane curves.  We showed further that periodic orbits, should they exist, must lie on this hypersurface.   Details and additional results appear in \cite{Shanbrom}.

This paper is dedicated to proving Theorem \ref{periodic}, which gives the existence of periodic orbits with $k$-fold rotational symmetry for any odd integer $k\geq 3$.  We employ the direct method from the calculus of variations, as in \cite{CM, Gordon} Numerical approximations of these orbits when $k=3$ are shown in Figures \ref{periodic1} and \ref{periodic2}.

\begin{figure}
\centering
\begin{tabular}{cc}
\hspace{-15pt}
\includegraphics[width=.60\textwidth]{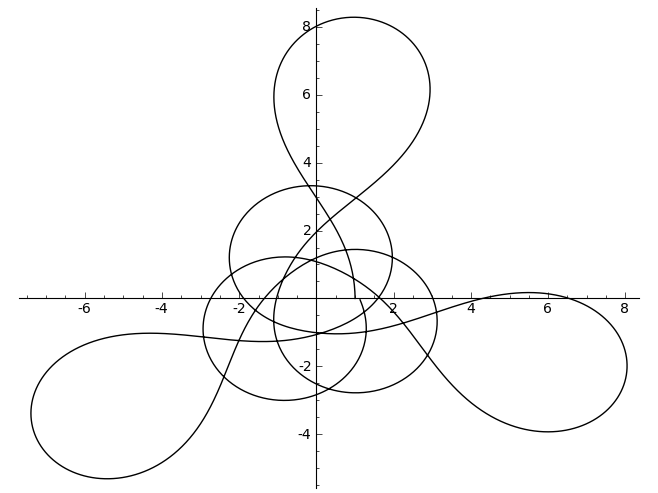}
\raisebox{3.6mm}{
\includegraphics[width=.40\textwidth]{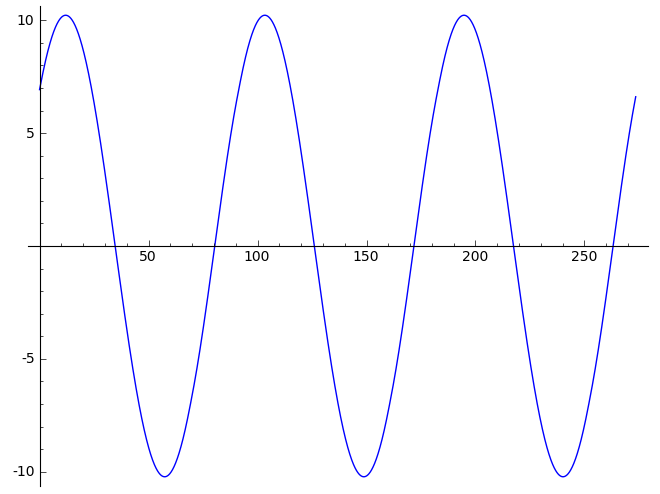}}
\end{tabular}
\caption{Projection of a periodic orbit to $xy$-plane (left).  The $z$-coordinate over time (right).}\label{periodic1}
\end{figure}

\section{The System}

\subsection{The Heisenberg Group}\label{Heisenberg}

 Consider $\mathbb R^3$ with standard $x,y,z$ coordinates,   endowed with the two vector fields
\[X=\frac{\partial}{\partial x}-\frac{1}{2}y\frac{\partial}{\partial z}, \qquad Y=\frac{\partial}{\partial y}+\frac{1}{2}x\frac{\partial}{\partial z}.\]
Then $X, Y$ span the canonical contact distribution $D$ on $\mathbb R^3$ with induced Lebesgue volume form.  Curves are called horizontal if they are tangent to $D$.
Declaring $X, Y$ orthonormal defines the standard sub-Riemannian structure on the Heisenberg group and yields the Carnot-Caratheodory metric $ds^2_{\mathbb H}=dx^2+dy^2$.  Geodesics are qualitatively helices: lifts of circles and lines in the $xy$-plane.  The horizontal constraint implies that the $z$-coordinate of a curve grows like the area traced out by the projection of the curve to the $xy$-plane. See Chapter 1 of \cite{Tour}.

We have
$$[X, Y] = \frac{\partial}{\partial z} =:Z $$
and $[X, Z] = [Y, Z] = 0$.  There are the commutation relations of the Heisenberg Lie algebra,
hence the name.  The Heisenberg group $\mathbb H$ is the simply connected Lie group
with Lie algebra the Heisenberg algebra and is diffeomorphic to  $\mathbb R^3$.
In $x,y,z$ coordinates the Heisenberg group law reads
\[(x_1,y_1,z_1)\cdot(x_2,y_2,z_2)=(x_1+x_2, y_1+y_2, z_1+z_2+\tfrac{1}{2}(x_1y_2-x_2y_1)).\]
Left multiplication is an isometry and the vector fields $X,Y$ are left invariant.

\begin{figure}
\centering
\includegraphics[height=80mm, keepaspectratio]{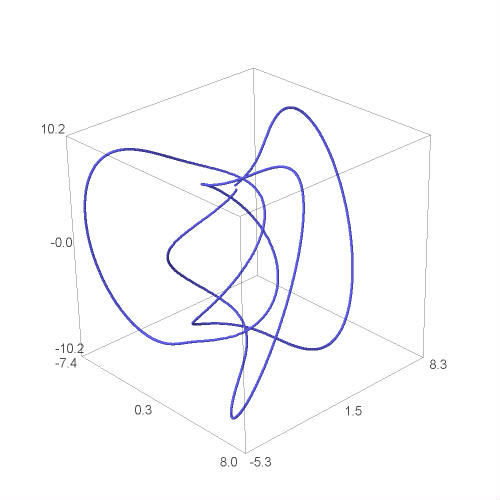}
\caption{A periodic orbit in three dimensions.}\label{periodic2}
\end{figure}

In the following we will denote the Heisenberg distance function by $d$:
\[d(p,q) = \text{the sub-Riemannian distance between points $p$ and $q$.}\]
There is no known explicit form of this function.
We will denote by $||\cdot||_{sR}$ the sub-Riemannian distance to the origin, and $||\cdot||_{\mathbb H}$ will denote the Heisenberg norm of a horizontal vector:
\begin{align*}
||p||_{sR}&=d(p, 0), \quad \quad p\in \mathbb H \\
||v||_{\mathbb H} &= \sqrt{\langle v, v \rangle}, \quad \quad v\in D_p.
\end{align*}
The sub-Riemannian gradient of a function $f \colon \mathbb H \to \mathbb R$, defined by the equation
$\langle \nabla_{sR}f, V \rangle=df(V),$ is the horizontal vector field
\[\nabla_{sR} f= X(f)X+Y(f)Y.\]

Now consider $T^*\mathbb H$ with canonical coordinates $(x,y,z,p_x,p_y,p_z)$. Then
\[P_X=p_x-\tfrac{1}{2}yp_z\quad \text{and} \quad P_Y=p_y+\tfrac{1}{2}xp_z\]
are dual momenta to $X$ and $Y,$ respectively.   Let
\[K=\tfrac{1}{2}(P_X^2+P_Y^2).\]
The function $K \colon T^*\mathbb H \to \mathbb R$ is known as the sub-Riemannian Hamiltonian, as the flow lines of its Hamiltonian vector field (symplectic gradient) are precisely the geodesics in $\mathbb H$.  As these geodesics can be thought of as trajectories of free particles, we will call the function $K$ our \emph{kinetic energy}.  While it is written here in coordinates, it can be defined canonically in terms of the cometric.

Finally, we observe that the Heisenberg group (like any Carnot group) admits dilations.  For any positive real number $\lambda$, define the map
\begin{align*} 
\delta_\lambda \colon \mathbb H &\to \mathbb H,\qquad
(x, y, z) \mapsto  (\lambda x, \lambda y, \lambda^2 z).\end{align*}
To say that this map is a dilation is to say that 
\[ d(\delta_{\lambda}(p), \delta_{\lambda}(q))=\lambda d(p,q), \quad \forall~ p,q\in \mathbb H.\]
This map lifts to a map on the cotangent bundle, which we also denote by $\delta_{\lambda}$:
\[ (x, y, z, p_x, p_y, p_z) \mapsto  (\lambda x, \lambda y, \lambda^2 z, \lambda^{-1} p_x, \lambda^{-1} p_y, \lambda^{-2} p_z).\]

\subsection{The Potential Energy}\label{potentialsection}
Recall the classical Kepler Problem on $\mathbb R^3$.  Let $r = \sqrt{x^2 + y^2  + z^2}$.  Then the Hamiltonian is 
\[ H = \tfrac{1}{2} (p_x ^2 +p_y ^2 + p_z ^2) - \frac{k}{r}.\] 
How do we characterize the potential $U=- \frac{k}{r}$?  The usual answer is that $U$ is (a constant times) the inverse of the distance function.  However, this characterization fails to provide guidance when we attempt to study the problem on spaces without an explicit distance function, such as the Heisenberg group.  A better answer is that, when $k=\frac{1}{4\pi}$, $U$ is the fundamental solution to the Laplacian on $\mathbb R^3$ (see \cite{Albouy}).  In other words, $U$ satisfies $\Delta U=\delta_0$, where $\delta_0$ is the Dirac delta function with source at 0.

Now consider the vector fields $X$ and $Y$ which form an orthonormal frame for the Heisenberg distribution.  Thinking of these as first-order differential operators, we define the Heisenberg sub-Laplacian to be the second-order subelliptic operator
$$ \Delta_{\mathbb H} = X^2+Y^2.$$
In \cite{Folland}, Folland found the fundamental solution to the Heisenberg sub-Laplacian $\Delta_{\mathbb H}$:
\[U=-\alpha\left((x^2+y^2)^2+\tfrac{1}{16}z^2\right)^{-\frac{1}{2}}. \]
In \cite{MS} we computed $\alpha=2/\pi$; for our purposes it suffices to leave $\alpha$ as a positive constant.  We recognize that this potential has a singularity at the origin $(x,y,z)=(0,0,0),$ but is smooth away from this point.  The singularity corresponds physically to our planet crashing into the sun, at which point the planet's potential energy is $-\infty$ and its kinetic energy is $+\infty$.  For this reason we will refer to a trajectory passing through the origin as a \emph{collision}.

For notational purposes, we set
\[ \mu=(x^2+y^2)^2+\tfrac{1}{16}z^2.\]
Note that Folland uses the notation $\rho=\mu^{1/4}$, which is homogeneous of degree 1 with respect to the dilation $\delta_{\lambda}$ defined in Section \ref{Heisenberg}; we will use $\rho$ later as a norm.  Then we can write
\[U=-\alpha\mu^{-1/2}=-\alpha\rho^{-2}.\]

\subsection{Equations of Motion}\label{equations}

We define our Hamiltonian $H$ in the usual way, as the sum of the kinetic and potential energies: $H=K+U$.  For this reason, we will often refer to $H$ as the \emph{energy}.  As our Hamiltonian has a singularity at the origin, this function is defined on the cotangent bundle of the Heisenberg group with the origin deleted: $$H \colon T^*(\mathbb H-\{(0,0,0)\}) \to \mathbb R.$$
Explicitly, we have
\begin{align}\label{H}
H&=\tfrac{1}{2}(P_X^2+P_Y^2)-\alpha\mu^{-\frac{1}{2}}\\
&=\tfrac{1}{2}(p_x-\tfrac{1}{2}yp_z)^2+\tfrac{1}{2}(p_y+\tfrac{1}{2}xp_z)^2-\tfrac{2}{\pi}\left((x^2+y^2)^2+\tfrac{1}{16}z^2\right)^{-\frac{1}{2}}.
\end{align}

\noindent Hamilton's equations read
\begin{align*}
\dot x=P_X, \qquad
\dot y&=P_Y, \qquad
\dot z=\frac{1}{2}xP_Y-\frac{1}{2}yP_X, \\
\dot p_x&= -\frac{1}{2}P_Yp_z-2\alpha x(x^2+y^2)\mu^{-3/2}, \\
\dot p_y&= \frac{1}{2}P_Xp_z-2\alpha y(x^2+y^2)\mu^{-3/2}, \\
\dot p_z&=-\frac{\alpha}{16}z\mu^{-3/2},
\end{align*} 
and the second derivatives of the position coordinates are
\begin{align*}
\ddot x&= \dot P_X= -P_Yp_z+\big(\tfrac{\alpha}{32}y z -2\alpha x(x^2+y^2) \big) \mu^{-3/2} \\
\ddot y &= \dot P_Y =P_Xp_z-\big(\tfrac{\alpha}{32}x z +2\alpha y(x^2+y^2) \big) \mu^{-3/2} \\
\ddot z &= \tfrac{1}{2}p_z(xp_x+yp_y)-\tfrac{\alpha}{64}(x^2+y^2)z\mu^{-3/2}.
\end{align*}

In order to prove Theorem \ref{periodic}, we consider our problem as a variational problem with subsidiary constraints.  
As usual, we define our Lagrangian $L \colon T\mathbb H \to \mathbb R$  as the difference of the kinetic and potential energies $L=K-U$.  Explicitly, we have
\[L(t,  q, {\dot q})= \tfrac{1}{2}\dot x^2 +  \tfrac{1}{2}\dot y^2+\alpha \Big( (x^2+y^2)^2+\tfrac{1}{16}z^2 \Big)^{-1/2}.
\]
(Here, and below, we write $ q=(x,y,z)$.)

Define the \emph{action} of an absolutely continuous path $\gamma \colon (0, T) \to \mathbb H$
by the functional
\[ A(\gamma) = \int_0^T L(t, \gamma, \dot \gamma) dt. \]
Here, we have the additional constraint that our solutions must be horizontal curves for our distribution.  
That is, solutions must lie on the zero set of the function  
\[G(t,  q, {\dot q})=\tfrac{1}{2}x\dot y - \tfrac{1}{2}y \dot x-\dot z.\]
The calculus of variations (see, for example, Section 12 of \cite{Gelfand}) tells us that if $\gamma$ is a minimum of the functional $A$ which also satisfies our constraint, then there exists a scalar $\lambda=\lambda(t)$ such that $\gamma$ minimizes the functional 
\[ A_{\lambda}(\gamma) = \int_0^T  L_{\lambda}(t, \gamma, \dot \gamma) dt,\]
where we have written $L_{\lambda}(t, q, \dot q)= L(t, q, \dot q)-\lambda(t)G(t, q, \dot q)$.

Next, we record the Euler-Lagrange equations 
\begin{equation}\label{EL}
\frac{\partial L_{\lambda}}{\partial q_i}- \frac{d}{dt}\bigg( \frac{\partial L_{\lambda}}{\partial \dot q_i}  \bigg)=0, 
\end{equation}
with $i=1,2,3$. 
A straightforward calculation yields:
\begin{align*}
\ddot x&=-\lambda\dot y - \tfrac{1}{2} \dot \lambda y - 2\alpha x(x^2+y^2)\mu^{-3/2} \\
\ddot y&=\lambda\dot x + \tfrac{1}{2} \dot \lambda x - 2\alpha y(x^2+y^2)\mu^{-3/2} \\
\dot \lambda &= -\tfrac{\alpha}{16}z\mu^{-3/2}.
\end{align*}
If we take $\lambda=p_z$, then these equations read
\begin{align}
\ddot x&=-\dot y p_z + \left(\tfrac{\alpha}{32}y z - 2\alpha x(x^2+y^2)\right)\mu^{-3/2} \\
\ddot y&=\dot x p_z- \left(\tfrac{\alpha}{32}x z + 2\alpha y(x^2+y^2)\right)\mu^{-3/2} \\
\dot p_z &=-\tfrac{\alpha}{16}z\mu^{-3/2}.
\end{align}
Note that these equations agree with the second derivatives given above, so that the Euler-Lagrange equations are indeed equivalent to Hamilton's equations.



\subsection{Properties}
We do not have many symmetries to work with.  In polar coordinates, the equations of motion are independent of the angular variable $\theta=\arctan(y/x)$.  Thus, the system enjoys rotational symmetry about the $z$-axis, and consequently, we find that angular momentum $p_{\theta}= xp_y-yp_x$ is conserved in time.  As this symmetry can be expressed as the invariance under an action of the compact one-dimensional Lie group $S^1$, symplectic reduction reduces the dimension of our system by one.
The only other symmetries are reflections, corresponding to actions of discrete Lie groups.

However, the dilations $\delta_{\lambda}$ correspond to an $\mathbb R^+$-action which is nearly a symmetry.  Our Hamiltonian is not preserved, but is homogeneous of degree $-2$:
\[\delta_\lambda \colon H \mapsto \lambda^{-2}H.
\]
More precisely, the dilations are generated by the function $J=xp_x+yp_y+2zp_z$ which satisfies $\dot J =2H$.  Thus, $J$ is an integral of motion for orbits with zero total energy.  This (with the fact that $\dot H=0$ by construction) implies that the system is integrable on the invariant hypersurface $\{H=0\}$ (see \cite{MS}).  This hypersurface is precisely where our periodic orbits live.

\begin{prop}
Periodic orbits must have zero energy.
\end{prop}

\begin{proof}
If $\gamma (t)=(x(t), y(t), z(t), p_x(t), p_y(t), p_z(t))$ satisfies $\gamma(0)=\gamma(T)$ for some $t=T$, then $J=xp_x+yp_y+2zp_z$ is also periodic; that is, $J(\gamma(0))=J(\gamma(T))$.  But we know that the time derivative of $J$ is constant, given by $\dot J=2H$.  Since $J$ cannot be monotonically increasing nor decreasing in time, we must have $\dot J=2H=0$.
\end{proof}

\section{Existence of Periodic Solutions}\label{existence}
In this section we prove our main theorem: there exist periodic orbits in the Kepler-Heisenberg problem.
These orbits were originally found by numerical experiment.  To prove their existence we employ the direct method in the calculus of variations, showing the existence of an action minimizing orbit with prescribed symmetry.  We prove the existence of solutions with $k$-fold rotational symmetry for any odd integer $k\geq 3$.
Approximations of one such orbit, with $k=3$, are shown in Figures \ref{periodic1} and \ref{periodic2}.  

\begin{theorem}\label{periodic}
Periodic solutions exist.  For any odd integer $k\geq 3$, there exists a periodic orbit with $k$-fold rotational symmetry about the $z$-axis.
\end{theorem}

\noindent\textbf{Proof of Theorem \ref{periodic}.}
We first sketch the structure of the proof, which uses the direct method from the calculus of variations.  For similar applications of this technique to celestial mechanics problems, see \cite{CM} and \cite{Gordon}. Each step below is presented in a separate subsection.\\

\noindent \textit{Step 1}: Choose a nice function space $\mathcal F_k$ whose members are closed loops enjoying the desired symmetry properties.  Choose a minimizing sequence of curves in $\mathcal F_k$ whose action approaches the infimum of the action restricted to $\mathcal F_k$.\\
\textit{Step 2}:  Using Arzela-Ascoli, show $\gamma_n$ has a $C^0$-convergent subsequence converging to some $\gamma_*$.  Using Banach-Alaoglu, show $\gamma_n \rightharpoonup \gamma_* \in \mathcal F_k$. \\
\textit{Step 3}: Show that $\gamma_*$ realizes the infimum of the action restricted to $\mathcal F_k$.  Use Fatou's Lemma and standard functional analysis.\\
\textit{Step 4}:  Prove that $\gamma_*$ does not suffer a collision.  Use the Hamilton-Jacobi equation.\\
\textit{Step 5}: Show $dA(\gamma_*)(e)=0$ for horizontal variations $e$.  Standard analysis gives $dA(\gamma_*)|_{\mathcal F_k} =0$, then use the Principal of Symmetric Criticality.\\
\textit{Step 6}: Show $\gamma_*$ satisfies the Euler-Lagrange equations, and consequently, Hamilton's equations.  This is the Principle of Least Action.

\subsection{The Function Space and a Minimizing Sequence}\label{functionspace}
For a curve $\gamma(t)=(x(t), y(t), z(t))$ in the Heisenberg group parametrized by the interval $[0,T]$, let $\tilde \gamma$ denote its projection to the $xy$-plane.  Let $k\geq 3$ be any odd\footnote{If $k$ is even, the two symmetry conditions force $z$ to be identically zero; such solutions are known (see \cite{MS}) to suffer collisions.} positive integer.
We will restrict our attention to horizontal curves satisfying the symmetry conditions
\begin{align*} \tag{S1}
 \gamma(t+ T/k)&=R_{2\pi/k}  \gamma(t)\\
 \tag{S2}
 z(t+ T/2)&=-z(t)
\end{align*}
where 
\[R_{2\pi/k}=\begin{bmatrix} \cos(2\pi/k) & -\sin(2\pi/k) & 0\\ \sin(2\pi/k) & \cos(2\pi/k) & 0\\0 & 0 & 1 \end{bmatrix}  \]
denotes rotation about the $z$-axis by $2\pi/k$ radians counterclockwise. 
Note that curves satisfying condition (S1) are necessarily periodic.  Also, note that the two symmetry conditions together give the $z$-coordinate symmetry
\[ z(t+mT/2k)=(-1)^mz(t),\]
for any odd integer $m$.

We will work in the function space
\[\mathcal F_k=\{\gamma \in H^{1}(S^1, \mathbb H) \ | \ \gamma \ \text{horizontal and satisfies}\  \text{(S1) and (S2)}  \},\] 
where $H^{1}(S^1, \mathbb H)=W^{1,2}(S^1, \mathbb H)$ is the completion of the space of all absolutely continuous paths in $\mathbb H$ whose derivative is square integrable\footnote{Here and in the sequel $S^1$ denotes the interval $[0,T]$ with endpoints identified.}. The usual $H^{1}(S^1, \mathbb R^3)$ norm is 
\[ ||\gamma||_{H^1}= \sqrt{\int_0^T(||\dot \gamma(t)||_E^2+||\gamma(t)||_E^2)dt },
\]
where $||\cdot||_E$ denotes the usual Euclidean norm on $\mathbb R^3$.  This norm endows $H^{1}(S^1, \mathbb R^3)$ with a Hilbert space structure.

To endow $\mathcal F_k$ with a Hilbert structure, we make the following identification:
\[ \mathcal L^2([0,T], \mathbb R^2) \times \mathbb H \cong \mathcal H:=\{\text{horizontal square-integrable paths in } \mathbb H\},\]
where $\mathbb H$ is identified with the vector space $\mathbb R^3$.
This isomorphism sends $((f, g), x_0)$ to the horizontal curve $\gamma$ which solves the initial value problem
\[ \dot \gamma= fX+gY, \quad \gamma(0)=x_0.
\]
The existence and uniqueness of the solution $\gamma$ is guaranteed by Theorem D.1 of \cite{Tour}.  This theorem also shows that this mapping is invertible for $x_0$ in some compact set.  Thus, we can think of $f,g$ as coordinates on the subspace of $\mathcal H$ consisting of all paths with a fixed starting point.  Consequently, $\mathcal F_k$ is equipped with a vector space structure.

Here we will endow $\mathcal F_k$ with a norm similar to the $H^1$ norm, but slightly modified for our purposes:
\begin{equation} \label{Fnorm}
||\gamma||_{\mathcal F_k}:= \sqrt{\int_0^T(||\dot \gamma(t)||^2_{\mathbb H} + ||\gamma(0)||^2_{E}) dt }.
\end{equation}

\begin{remark}
Here we have written $||\dot \gamma(t)||^2_{\mathbb H} = \langle \dot \gamma(t), \dot \gamma(t) \rangle$.  Note that the ``norm''
\[||\gamma||= \sqrt{\int_0^T(||\dot \gamma(t)||^2_{\mathbb H} + ||\gamma(0)||^2_{sR}) dt }
\]
is not a vector space norm.  It satisfies $|| \delta_{\lambda}\gamma ||=|\lambda|\ || \gamma ||$ just as $||p||_{sR}=d(p, 0)$ satisfies $|| \delta_{\lambda}(p) ||_{sR}=|\lambda|\ || p ||_{sR}$ for $p\in \mathbb H$.
However, it does induce a genuine distance function $d(\gamma_1, \gamma_2)=|| \gamma_1-\gamma_2 ||$.  Also note that the Euclidean topology on $\mathbb H$ is the same as the topology induced by the Carnot-Caratheodory metric, so a set in $\mathbb H$ is bounded in the Carnot-Caratheodory metric if and only if it is bounded in the Euclidean metric. 
Finally, we shall also make use of the standard $\mathcal L^2$ norm
\[||\dot \gamma||^2_{\mathcal L^2} = \sqrt{\int_0^T||\dot \gamma(t)||^2_{\mathbb H}dt}. \]

\end{remark}

\begin{prop}[Coercivity] \label{coercivity}
The squared length of $\tilde \gamma$ is bounded above by twice the action of $\gamma$. 
\end{prop}

\begin{proof}
Let $\tilde \gamma$ denote the projection of $\gamma(t)=(r(t), \theta(t), z(t))$ to the $r, \theta$ plane, where $r=\sqrt{x^2+y^2}$ and $\theta=\arctan(y/x)$ are standard polar coordinates on the plane.  Then we compute
\begin{align*}
A(\gamma) &= \int_0^T L(t, \gamma, \dot \gamma) dt \\
&= \int_0^T \big(\tfrac{1}{2}\dot r^2 + \tfrac{1}{2} r^2 \dot \theta^2 + \alpha(r^4+\tfrac{1}{16}z^2)^{-1/2}\big)dt \\
&\geq\int_0^T \big(\tfrac{1}{2}\dot r^2 + \tfrac{1}{2} r^2 \dot \theta^2 \big) dt \\
&= \tfrac{1}{2} \int^T_0 (\dot x^2 +\dot y^2) dt \\
& \geq \tfrac{1}{2} \left( \int_0^T \sqrt{\dot x^2 +\dot y^2} dt\right)^2 \\
&= \tfrac{1}{2} \big(l ( \gamma) \big)^2,
\end{align*}
where we used the Cauchy-Schwarz inequality in $\mathcal L^2(\mathbb R)$.  
\end{proof}

Now suppose $\{\gamma_n\}_{n\in \mathbb N}$ is a minimizing sequence in $\mathcal F_k$.  That is, suppose 
\[ \lim_{n\to \infty} A(\gamma_n)= \inf_{\gamma \in \mathcal F_k} A(\gamma). \]
We may discard finitely many terms of the sequence and assume that there exists some large $M>0$ such that 
\[A(\gamma_n) \leq M\] 
for all $n$.  Note that the previous Proposition implies that the lengths $l(\gamma_n)$ are bounded; specifically, 
$l(\gamma_n) \leq \sqrt{2M}$
for all $n$.



\subsection{The Potential Solution}
We will denote the usual Euclidean norm by $||\cdot||_E$ and the corresponding Euclidean distance function by $d(\cdot, \cdot)_E$.  Note that  $d(\cdot, \cdot)_E\leq d(\cdot, \cdot)$ in general, but that both the Euclidean and sub-Riemannian distances agree when restricted to the $xy$-plane.

\begin{prop}\label{projbdd}
The projections $\tilde \gamma_n$ are uniformly bounded.
\end{prop}

\begin{proof}
Since $\gamma_n \in \mathcal F_k$, it is horizontal, so the length of $\gamma_n$ is equal to that of $\tilde \gamma_n$.  Also, condition (S1) implies that \[z(0)=z(T/k)=z(2T/k)=\cdots=z((k-1)T/k)=z(T)\] and thus that the $k$ points 
\[\gamma_n(0)=\gamma_n(T),\ \gamma_n(T/k),\ \gamma_n(2T/k), \dots,\ \gamma_n((k-1)T/k)\]
form a regular $k$-gon in the plane $z=z(0)$ centered at the point $(0,0,z(0))$.  See Figure \ref{polygons} for a rendering of the $k=5$ case. Then we have:
\begin{align*} 
l(\gamma_n) 
&\geq d(\gamma_n(0), \gamma_n(\tfrac{T}{k})) + d(\gamma_n(\tfrac{T}{k}), \gamma_n(\tfrac{2T}{k})) + \cdots +d(\gamma_n(\tfrac{(k-1)T}{k}), \gamma_n(T)) \\
&\geq d_E(\gamma_n(0), \gamma_n(\tfrac{T}{k})) + d_E(\gamma_n(\tfrac{T}{k}), \gamma_n(\tfrac{2T}{k})) + \cdots +d_E(\gamma_n(\tfrac{(k-1)T}{k}),\gamma_n(T)) \\
&= d_E(\tilde\gamma_n(0), \tilde\gamma_n(\tfrac{T}{k})) + d_E(\tilde\gamma_n(\tfrac{T}{k}), \tilde\gamma_n(\tfrac{2T}{k})) + \cdots +d_E(\tilde\gamma_n(\tfrac{(k-1)T}{k}), \tilde\gamma_n(T)) \\
&= C||\tilde \gamma_n(0)||_E \\
&= C||\tilde \gamma_n(0)||_{sR},
\end{align*}
where $C=2k\sin(2\pi/k)$ and the penultimate equality is given by the usual perimeter of a regular $k$-gon inscribed in a circle of radius $||\tilde \gamma_n(0)||_E.$

Then since $l(\gamma_n)=l(\tilde \gamma_n)$, we find
\begin{align*}
||\tilde\gamma_n(t)||_{sR} &\leq  ||\tilde\gamma_n(0)||_{sR} + l(\tilde\gamma_n) \\
&\leq (\tfrac{1}{C}+1)l(\gamma_n) \\ &\leq (\tfrac{1}{C}+1)\sqrt{2M}.\end{align*}

\end{proof}

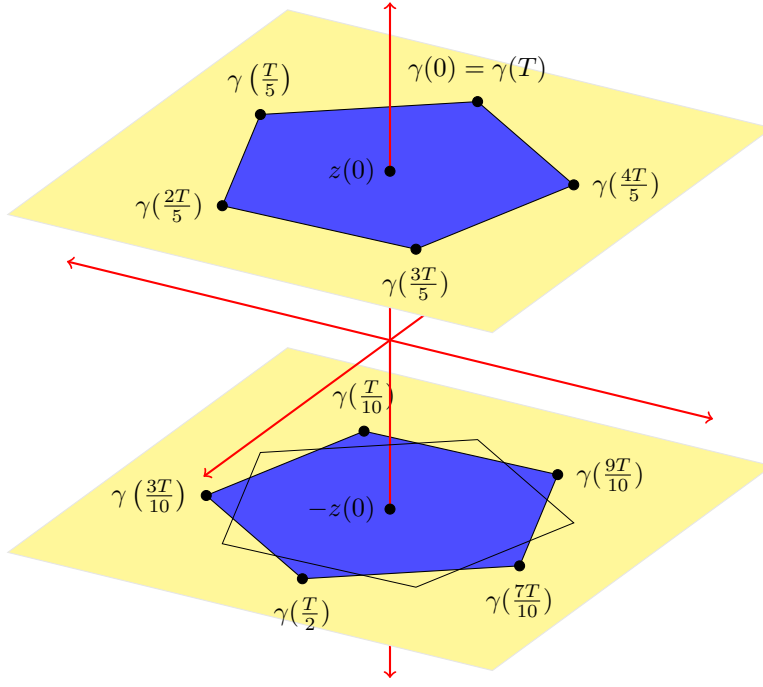
\begin{figure}
\begin{center}

\tdplotsetmaincoords{65}{30}
\begin{tikzpicture}[scale=1.24, tdplot_main_coords]

\draw[thick,red] (0,0,0) -- (0,0,2);
\draw[thick,red,<-] (0,4,0) -- (0,0,0);
\draw[thick,red] (0,0,0) -- (0,0,-2);
\draw[thick,red, ->] (0,0,-2) -- (0,0,-4);

\draw[draw=gray!20, fill=yellow!50, opacity=100] (3,3,2)--(-3,3,2)--(-3,-3,2)--(3,-3,2)--(3,3,2)--cycle;

\draw[draw=gray!20, fill=yellow!50, opacity=100] (3,3,-2)--(-3,3,-2)--(-3,-3,-2)--(3,-3,-2)--(3,3,-2)--cycle;

\draw[thick,red,<->] (-4,0,0) -- (4,0,0);
\draw[thick,red,->] (0,0,0) -- (0,-4,0);

\coordinate (up) at (0,0,2);

\begin{scope}[shift=(up)]
\foreach \x in {1,2,...,5}{
\coordinate (U-\x) at (\x*72+20:2);
}
\end{scope}

\draw[fill=blue!70,opacity=100] (U-1)--(U-2)--(U-3)--(U-4)--(U-5)--cycle;

\node[draw,circle,fill,scale=0.4] at (U-1){};
\node[label={$\gamma(0)=\gamma(T)$}] at (U-1) {};

\node[draw,circle,fill,scale=0.4] at (U-2){};
\node[label=$\gamma\left(\frac{T}{5}\right)$] at (U-2) {};

\node[draw,circle,fill,scale=0.4] at (U-3){};
\node[label=left:$\gamma(\tfrac{2T}{5})$] at (U-3) {};

\node[draw,circle,fill,scale=0.4] at (U-4){};
\node[label=below:$\gamma(\tfrac{3T}{5})$] at (U-4) {};

\node[draw,circle,fill,scale=0.4] at (U-5){};
\node[label=right:$\gamma(\tfrac{4T}{5})$] at (U-5) {};

\coordinate (down) at (0,0,-2);
\begin{scope}[shift=(down)]
\foreach \x in {1,2,...,5}{
\coordinate (D-\x) at (\x*72+56:2);
}
\end{scope}

\draw[fill=blue!70,opacity=100] (D-1)--(D-2)--(D-3)--(D-4)--(D-5)--cycle;

\node[draw,circle,fill,scale=0.4] at (D-1){};
\node[label=$\gamma(\frac{T}{10})$] at (D-1) {};

\node[draw,circle,fill,scale=0.4] at (D-2){};
\node[label=left:$\gamma\left(\frac{3T}{10}\right)$] at (D-2) {};

\node[draw,circle,fill,scale=0.4] at (D-3){};
\node[label=below:$\gamma(\tfrac{T}{2})$] at (D-3) {};

\node[draw,circle,fill,scale=0.4] at (D-4){};
\node[label=below:$\gamma(\tfrac{7T}{10})$] at (D-4) {};

\node[draw,circle,fill,scale=0.4] at (D-5){};
\node[label=right:$\gamma(\tfrac{9T}{10})$] at (D-5) {};

\draw[thick,red,->] (0,0,2) -- (0,0,4);
\draw[thick,red] (0,0,0) -- (0,0,-2);
\node[draw, circle, fill, scale=0.4, label=left:$z(0)$] at (0,0,2) {};
\node[draw, circle, fill, scale=0.4, label=left:$-z(0)$] at (0,0,-2) {};

\begin{scope}[shift=(down)]
\foreach \x in {1,2,...,5}{
\coordinate (U-\x) at (\x*72+20:2);
}
\end{scope}

\draw[very thin]
 (U-1)--(U-2)--(U-3)--(U-4)--(U-5)--cycle;

\end{tikzpicture}
\end{center}

\caption{Symmetries of a path in $\mathcal F_5$.}
\label{polygons}
\end{figure}

\begin{lemma}\label{zbound}
Suppose $\gamma \colon [0,S] \to \mathbb H$ is horizontal.  Suppose $c$, the projection of $\gamma$ to the $xy$-plane, satisfies $||c(t)||_E\leq R$ for some $R>0$ and all $t \in [0,S]$.  Then 
\[|z(S)-z(0)| \leq \tfrac{1}{2}Rl(\gamma).
\]
\end{lemma}

\begin{proof}
Let $I(v_1, v_2)=(v_2, -v_1).$  Note that $||I(\dot c)||_{\mathbb H}=||\dot c||_{\mathbb H}$.  Without loss of generality, we may assume that $c$ has constant speed $v$.  Then by the horizontal condition $\tfrac{1}{2}x\dot y - \tfrac{1}{2}y \dot x=\dot z$ and the Cauchy-Schwarz inequality, 
\begin{align*}
|z(S)-z(0)| 
&= \tfrac{1}{2}|\int_0^S (x\dot y-y\dot x)dt|\\
&= \tfrac{1}{2}|\int_0^S \langle c, I(\dot c) \rangle_E dt|\\
&\leq \tfrac{1}{2}\sqrt{\int_0^S ||c(t)||_E^2 dt} \sqrt{\int_0^S ||I(\dot c(t))||_E^2 dt} \\
&= \tfrac{1}{2}\sqrt{\int_0^S ||c(t)||_E^2 dt} \sqrt{\int_0^S ||\dot c(t)||_E^2 dt} \\
&\leq\tfrac{1}{2}\sqrt{R^2S}\sqrt{v^2S}\\
&=\tfrac{1}{2}Rl(\gamma).
\end{align*}

\end{proof}

\begin{prop} \label{bounded}
The set $\{\gamma_n\}$ is uniformly bounded.
\end{prop}

\begin{proof}
By Proposition \ref{projbdd}, the curves $\gamma_n$ satisfy the hypothesis of Lemma \ref{zbound}: we can take $R=(\tfrac{1}{C}+1)\sqrt{2M}$.  Take $S=T/2$ and denote the $z$-components of the curves $\gamma_n$ by $z_n$.  Then the Lemma, with the fact that $z_n(0)=-z_n(T/2)$, implies
\[|2z_n(0)| \leq \tfrac{1}{2}Rl(\gamma_n|_{[0, T/2]}) \leq \tfrac{1}{2}Rl(\gamma_n). \]
Then we find
\begin{align*}
|z_n(t)| &\leq |z_n(0)|+l(\gamma_n) \\
&= (\tfrac{1}{4}R+1)l(\gamma_n) \\
&\leq  (\tfrac{1}{4}R+1)\sqrt{2M}.
\end{align*}
Thus, the family $z_n(t)$ is uniformly bounded.
Since the projections $\tilde \gamma_n$ and the $z$-coordinates $z_n$ are uniformly bounded, so are the curves $\gamma_n$.
\end{proof}

\begin{lemma}\label{holder}
If $\gamma \in \mathcal F_k$ then $\gamma$ is H\"{o}lder continuous with H\"{o}lder exponent $\frac{1}{2}$.
\end{lemma}

\begin{proof}
This is a version of the Sobolev embedding theorem.  Using the Cauchy-Schwarz inequality with $f
=||\dot \gamma(t)||_{\mathbb H}$ and $g=1$, one has
\begin{align*}
d(\gamma(r),\gamma(s)) &\leq l( \gamma|_{[r,s]}) \\
&=\int_r^s ||\dot \gamma(t)||_{\mathbb H} dt \\
&\leq \sqrt{|r-s|\int^T_0||\dot \gamma(t)||_{\mathbb H}^2dt} \\
&=|r-s|^{1/2}\ ||\dot \gamma||_{\mathcal L^2}.
\end{align*}
This shows that $\gamma$ satisfies the H\"{o}lder condition with exponent $1/2$ and coefficient $||\dot \gamma||_{\mathcal L^2}.$  
\end{proof}

\begin{lemma}\label{L2bdd}
The norms $||\dot \gamma_n ||_{\mathcal L^2}$ are bounded. 
\end{lemma}

\begin{proof}
We have 
\begin{align*}
||\dot \gamma_n ||^2_{\mathcal L^2} &= \int_0^T ||\dot \gamma_n ||^2_{\mathbb H}dt \\
&= \int_0^T2K(\gamma_n)dt \\
&\leq 2\int_0^T (K(\gamma_n)+U(\gamma_n))dt\\
&= 2 A(\gamma_n)\\
& \leq 2M.
\end{align*}
\end{proof}

\begin{prop}\label{equicont}
The family $\{ \gamma_n\}$ is equicontinuous.
\end{prop}

\begin{proof}
By Lemma \ref{holder}, we know
\[d(\gamma_n(r),\gamma_n(s)) \leq \sqrt{|r-s|}\ ||\dot \gamma_n||_{\mathcal L^2}.
\]
Now choose some $\epsilon >0$ and let $\delta=(\epsilon/\sqrt{2M})^2$.  Then if $|r-s|<\delta$, by  Lemma \ref{L2bdd}, for all $n$ we have
\begin{align*}
d(\gamma_n(r),\gamma_n(s)) 
&<\frac{\epsilon}{\sqrt{2M}}\ ||\dot \gamma_n||_{\mathcal L^2} 
\leq \epsilon.
\end{align*}

\end{proof}

\begin{prop} 
There is a subsequence $\{\gamma_{n_j}\}$ converging uniformly to some  $\gamma_*$. 
\end{prop}

\begin{proof}\label{unif}
By Propositions \ref{bounded} and \ref{equicont}, the sequence $\{\gamma_n\}$ is uniformly bounded and equicontinuous.
The Arzela-Ascoli theorem guarantees the existence of such a subsequence, which implies that $\gamma_*$ is continuous.  
\end{proof}

\begin{prop}\label{weak}
There is a subsequence $\{\gamma_{n_{j_i}}\}$ which converges weakly in $\mathcal F_k$ to  $\gamma_* \in \mathcal F_k$. 
\end{prop}

\begin{proof}
From Proposition \ref{bounded} we know that the norms  $||\gamma_n(0)||_{E}$ are bounded; also, by Lemma \ref{L2bdd}, the norms $||\dot \gamma_n||_{\mathcal L^2}$ are bounded. 
Thus, the norms
\[||\gamma_n||^2_{\mathcal F_k}=||\dot \gamma_n||^2_{\mathcal L^2}+T||\gamma_n(0)||^2_{E}
\] are bounded as well.
Then the Banach-Alaoglu theorem guarantees the existence of another subsequence which converges weakly to $\gamma_*$ so that this limiting curve is absolutely continuous.  
It is clear that this curve must satisfy the horizontal and symmetry conditions.
\end{proof}

\begin{remark}In the following we will re-index so that the sequence $\{\gamma_n\}$ converges to $\gamma_*$ both weakly and uniformly.
\end{remark}

\subsection{Minimization} We now prove that $\gamma_*$ minimizes the action functional.
\begin{lemma} One has
\[\int_0^T\langle \dot \gamma_n, \dot \gamma_* \rangle dt\to \int_0^T ||\gamma_*||^2_{\mathbb H}dt. \] 
\end{lemma}

\begin{proof}
From Proposition \ref{weak}, we know that $\{\gamma_n\}$ converges weakly to $\gamma_*$.
By definition, we have that 
$\phi(\gamma_n) \to \phi(\gamma_*) $
for any $\phi \in \mathcal F_k^*$.  In particular,  the functional
\[\phi_c(\gamma) = \int_0^T \langle\dot \gamma(t),\dot c(t) \rangle dt \]
is indeed an element of $\mathcal F_k^*$ for any $c \in \mathcal F_k$.  To see that $\phi_c$ is a bounded operator, note that the Cauchy-Schwarz inequality in $\mathcal L^2$ gives
\[\phi_c(\gamma)=\langle \dot c, \dot\gamma \rangle_{\mathcal L^2} \leq ||\dot c ||_{\mathcal L^2} ||\dot \gamma ||_{\mathcal L^2}= C||\dot \gamma ||_{\mathcal L^2}.
\]
We may therefore choose $c=\gamma_*$, and the Lemma follows. 

Alternatively, one can see that $\phi_{\gamma_*}$ is bounded on the sequence $\{\gamma_n\}$ by writing
\begin{equation}\label{bddoperator}
\phi_{\gamma_*}(\gamma_n)=\langle \dot \gamma_n, \dot \gamma_* \rangle_{\mathcal F_k}-\langle  \gamma_n(0),  \gamma_*(0) \rangle_E,
\end{equation}
where $\langle \cdot, \cdot \rangle_E$ denotes the Euclidean inner product on $\mathbb R^3$. Then, considering the right hand side of \eqref{bddoperator}, note that one can control the first term by weak convergence, and the second by uniform convergence.
\end{proof}

\begin{prop}\label{inf}
Our limiting curve $\gamma_*$ realizes the infimum of the action:
$$A(\gamma_*)=\inf_{\gamma \in \mathcal F_k}A(\gamma).$$
\end{prop}

\begin{proof}
Since
\[0 \leq \int_0^T ||\dot\gamma_n-\dot \gamma_*||^2_{\mathbb H}dt = \int_0^T (||\dot\gamma_n||^2_{\mathbb H}+||\dot\gamma_*||^2_{\mathbb H}-2\langle \dot \gamma_n, \dot \gamma_* \rangle) dt,
\]
we have 
\[\int_0^T 2\langle \dot \gamma_n, \dot \gamma_* \rangle dt
\leq
\int_0^T (||\dot\gamma_n||^2_{\mathbb H}+||\dot\gamma_*||^2_{\mathbb H}) dt.
\]
Taking the limit inferior of both sides and using the previous Lemma yields
\[\int_0^T ||\dot\gamma_*||^2_{\mathbb H} dt
\leq
\liminf \int_0^T ||\dot\gamma_n||^2_{\mathbb H} dt.
\]
We may rewrite this last inequality as
\begin{equation}\label{kinetic}
\int_0^T K(\gamma_*(t)) dt
\leq
\liminf \int_0^T K(\gamma_n(t)) dt.
\end{equation}
Now, we know $\gamma_n(t) \to \gamma_*(t)$ uniformly.  Since our potential $U$ is continuous (except at the origin), we have that $U(\gamma_n(t)) \to U(\gamma_*(t))$ uniformly almost everywhere.  Then Fatou's Lemma implies
\begin{equation}\label{potential}
\int_0^T U(\gamma_*(t)) dt
\leq \liminf \int_0^T U(\gamma_n(t)) dt.
\end{equation}
Adding \eqref{kinetic} and \eqref{potential} gives
\[
A(\gamma_*)=\int_0^T L(\gamma_*(t)) dt
\leq
\liminf \int_0^T L(\gamma_n(t)) dt=\liminf A(\gamma_n).
\]
But $\{\gamma_n\}$ is a minimizing sequence, and $\gamma_* \in \mathcal F_k$ so
\[
A(\gamma_*)=\liminf A(\gamma_n)=\inf_{\gamma \in \mathcal F_k}A(\gamma).
\]

\end{proof}

\subsection{Avoiding Collision}
We will show that a curve in $\mathcal F_k$ suffering a collision necessarily has infinite action.  Without loss of generality, we may assume the collision occurs at time $t=0$.


Let $H_g \colon T^*\mathbb H \to \mathbb R$ denote the Hamiltonian generating geodesic flow on the Heisenberg group (which is also our kinetic energy, $K$).  In other words, let
\[H_g(q,p)=\frac{1}{2}(P_X^2+P_Y^2)=\frac{1}{2}(p,p),\]
where $(\cdot, \cdot)$ is the cometric induced by the inner product $\langle \cdot, \cdot \rangle$.
Then let  
\[S(q,t)=\inf \int_0^tL_g(\gamma(t), \dot \gamma(t))dt = \inf \int_0^t\tfrac{1}{2}||\dot \gamma(s)||^2_{\mathbb H}ds,
\]
where the infimum is taken over all paths $\gamma$ connecting 0 to $q$ in time $t$.  Here, the Lagrangian $L_g$ is related to the Hamiltonian $H_g$ by the Legendre transform.  Note that the function $S$ is known as \emph{Hamilton's generating function} or the \emph{action} in mechanics, and the \emph{value function} in optimal control.
Then the Hamilton-Jacobi equation (see \cite{Foundations} or \cite{Arnold1}) says
\begin{equation}\label{Ham-Jac}
H_g(q, dS)=-\frac{\partial S}{\partial t}
\end{equation}
at all points $(q, t)$ where $S$ is differentiable.
As the next Lemma shows, $S$ is differentiable at points $(q,t)$ where $t\neq 0$ and the function $f(q)=d(q,0)$ is smooth at $q$.  The latter holds if, for the minimizing geodesic $\gamma$ connecting $q$ to the origin, $\gamma$ contains no conjugate or cut points.  In the Heisenberg group, there are no cut points, and the locus of points conjugate to the origin consists of the $z$-axis (see \cite{Brockett}).
Thus, the Hamilton-Jacobi equation \eqref{Ham-Jac} holds almost everywhere: for all points $(q,t)$ such that $t \ne 0$ and $q$ does not lie on the $z$-axis. 


\begin{lemma}\label{computeS}
We can express the generating function $S$ as 
\[
S(q,t)=\frac{||q||^2_{sR}}{2t}.
\]
\end{lemma}

\begin{proof}
Suppose $\gamma \colon [0,t] \to \mathbb H$, with $\gamma(0)=0$ and $\gamma(t)=q$.  Then the Cauchy-Schwarz inequality gives
\[\int_0^t||\dot \gamma(s)||_{\mathbb H}ds \leq \sqrt{\int_0^t||\dot \gamma(s)||^2_{\mathbb H}ds}\sqrt{\int_0^tds},
\]
with equality if and only if the speed $||\dot \gamma(s)||_{\mathbb H}$ is constant.  We recognize the left-hand side as the length $l(\gamma)$.  Now suppose $\gamma$ has constant speed, so we may rewrite this (in)equality as
\[l(\gamma)=\sqrt{t}~ \sqrt{\int_0^t||\dot \gamma(s)||^2_{\mathbb H}ds}, 
\]
or, 
\[\int_0^t\tfrac{1}{2}||\dot \gamma(s)||^2_{\mathbb H}ds=\frac{l(\gamma)^2}{2t}.
\]
Finally, taking the infimum (over all such $\gamma$) of both sides yields the desired result.

\end{proof}

\begin{lemma}\label{speed}
Suppose $\gamma \colon [0,T] \to \mathbb H$ is horizontal and $\gamma(0)=0$.  Then 
\[\frac{d}{dt}||\gamma(t)||_{sR} \leq ||\dot \gamma(t)||_{\mathbb H} \]
almost everywhere.
\end{lemma}

\begin{proof}
Let $R(q)=||q||_{sR}$ for sake of notation.  Then, by Lemma \ref{computeS}, we have
\[S(q,t)=\frac{R^2(q)}{2t},\]
so that 
\[dS=\frac{R(q)dR}{t}\]
and 
\[\frac{\partial S}{\partial t}=-\frac{R^2(q)}{2t^2}.\]
Choosing $(q,t)$ such that $R(q)=t$ implies $dS=dR$ and 
\[\frac{\partial S}{\partial t}=-\frac{1}{2}.\]
Thus, the Hamilton-Jacobi equation reads\footnote{An optimal control theoretic version of this result can be found in Chapter 1, Section 9, of \cite{Pont}.}
\[\frac{1}{2}(dR,dR)=\frac{1}{2},\]
which is equivalent to
\[||\nabla_{sR} R||_{\mathbb H}=1,\]
and this holds almost everywhere.
Finally, we can employ the chain rule and the Cauchy-Schwarz inequality in the Heisenberg group to obtain:
\begin{align*}
\frac{d}{dt}||\gamma(t)||_{sR} &= \frac{d}{dt} R(\gamma(t)) \\
&\leq |\frac{d}{dt} R(\gamma(t))| \\
&= |\langle \nabla_{sR} R(\gamma(t)), \dot \gamma(t)\rangle |\\
&\leq ||\nabla_{sR} R(\gamma(t))||_{\mathbb H}||\dot\gamma(t)||_{\mathbb H} \\
&=||\dot\gamma(t)||_{\mathbb H}
\end{align*}
almost everywhere.
\end{proof}

\begin{prop}
Suppose $\gamma \in \mathcal F_k$ and $\gamma(0)=0$.  Then $A(\gamma)=\infty$.
\end{prop}

\begin{proof} Recall that $\rho=((x^2+y^2)^2+\tfrac{1}{16}z^2)^{1/4}$, so that $U=\rho^{-2}$.
Since the Heisenberg sphere is homeomorphic to the Euclidean sphere,
the standard argument which shows that any two norms on $\mathbb R^n$ are Lipshitz equivalent
shows that $\rho$ and $||\cdot||_{sR}$ are Lipshitz equivalent: there exist positive constants $c$ and $C$
such that $$c \rho(x,y,z) < ||(x,y,z)||_{sR} < C \rho(x,y,z)$$ for $(x,y,z) \ne 0$. 
This, with the general fact that  $a^2+b^2 \geq 2ab$, gives
\begin{align*}
A(\gamma)&=\int_{\gamma}L\\
&= \int_{\gamma}\tfrac{1}{2}(\dot x^2 +\dot y^2)+\rho^{-2}\\
&\geq \int_{0}^T\Big(\tfrac{1}{2}||\dot \gamma||^2_{\mathbb H}+\frac{c_1}{||\gamma||^2_{sR}}\Big)dt\\
&\geq c_2\int_0^T \frac{||\dot \gamma||_{\mathbb H}}{||\gamma||_{sR}} dt.
\end{align*}
But this last integrand is non-negative, so the value of its integral decreases when taken over a sub-interval of $[0,T]$.  In particular, the value of the integral is smaller over the interval $[0, \epsilon]$ for small $\epsilon$.  This fact, along with Lemma \ref{speed}, implies 
\begin{align*}
A(\gamma)&\geq c_2\int_0^T \frac{||\dot \gamma||_{\mathbb H}}{||\gamma||_{sR}} dt \\
&\geq c_2\int_0^{\epsilon} \frac{||\dot \gamma||_{\mathbb H}}{||\gamma||_{sR}} dt \\
&\geq c_3\int_0^{\epsilon} \frac{\frac{d}{dt}|| \gamma||_{sR}}{||\gamma||_{sR}} dt \\
&= c_3\int_0^{u({\epsilon})}\frac{du}{u} \\
&= \lim_{a\to 0} c_3 \log u |_a^{u({\epsilon})}\\
&= \infty.
\end{align*}
Here, we have made the substitution $u(t)=||\gamma(t)||_{sR}$ with $u(0)=0$.
\end{proof}

\begin{cor}\label{nocollision}
Our curve $\gamma_*$ does not suffer a collision.
\end{cor}

\begin{proof}
From Proposition \ref{inf}, we know $A(\gamma_*)$ equals the infimum of the action restricted to $\mathcal F_k$, which is finite.
\end{proof}

\subsection{A Critical Point of the Action} Here we prove   $dA(\gamma_*)=0$ on horizontal variations.

\begin{lemma}
The action functional $A$ is differentiable at any curve which avoids collision.
\end{lemma}

\begin{proof}
The proof is straightforward.  Let $e(t)=\frac{d}{dh}c(t,h)|_{h=0} $ be a variation of a horizontal path $c(t)$.  We compute the derivative of $A$ to be

\begin{align}
dA(c)(e) &= \frac{d}{dh} A(c(h))|_{h=0}\nonumber \\
&= \label{actionderivative} \int_0^T \langle \dot c(t), \dot e(t) \rangle_E + \langle \nabla U(c(t)), e(t) \rangle_E dt.
\end{align}
Here, $\nabla$ denotes the Euclidean gradient operator and both inner products are Euclidean.  In the first, $\dot c$ and $\dot e$ are horizontal so this is the same as the sub-Riemmanian inner product; in the second, neither term need be horizontal.
Then the expression \eqref{actionderivative} shows that $dA(c)$ exists and is indeed continuous, so long as $c$ does not pass through the origin, where $U$ has a singularity.
\end{proof}

\begin{prop}\label{critpt}
We have that $dA(\gamma_*)(e)=0$ for any $e \in \mathcal F_k$. 
\end{prop}

\begin{proof}
By Corollary \ref{nocollision}, $\gamma_*$ does not pass through the origin, so by the previous Lemma, $dA(\gamma_*)$ exists and is continuous.  By Proposition \ref{inf}, $\gamma_*$ is a local minimum of $A|_{\mathcal F_k}$.  The standard argument then gives our result.

\end{proof}

To combine the results obtained thus far, we need the following lemma, due to R. Palais (\cite{Palais}).

\begin{lemma}[Principle of Symmetric Criticality]
Let $\Gamma$ be a finite group acting on Hilbert space $V$ and let $V^{\Gamma}$ denote the fixed points of this action.  Suppose  $f \colon V \to \mathbb{R}$ is $\Gamma$-invariant, and that $f|_{V^{\Gamma}}$ has a critical point at $x_0 \in V$.  Then $x_0$ is also a critical point for $f$.
\end{lemma}

\begin{prop}
We have $dA(\gamma_*)(e)=0$ for any horizontal $e \in H^1([0, T], \mathbb H)$.
\end{prop}

\begin{proof}
Let $V$ be the space of horizontal paths in $H^1([0,T], \mathbb H)$, and let $\Gamma$ be the group $\mathbb Z/\mathbb Z_2 \times \mathbb Z/\mathbb Z_k$, $k\geq 3$ an odd integer, whose action is given as follows\footnote{An additional application of this argument allows us to restrict our attention to periodic orbits, so that the domains of these curves are well-defined.}: 
\begin{align*}
\mathbb Z/\mathbb Z_2 \times \mathbb Z/\mathbb Z_k &= \langle (\sigma, \tau)\ | \ (\sigma^2, \tau^k)=1\rangle\\
\sigma \cdot (x(t), y(t), z(t)) &=(x(t-T/2), y(t-T/2), -z(t-T/2))\\
\tau \cdot (x(t), y(t), z(t)) &= R_{2\pi/k}(x(t-T/k), y(t-T/k), z(t-T/k))
\end{align*}
Then $ V^{\Gamma}$ is precisely the function space $\mathcal F_k$.
Proposition \ref{critpt} and the previous Lemma imply the result.
\end{proof}

\subsection{Satisfaction of the Equations of Motion}
As shown in the previous section, $\gamma_*$ is a critical point of the action functional restricted to horizontal variations.  According to the Principal of Least Action,  $\gamma_*$ should satisfy the equations of motion.  More precisely, we will employ the standard argument from the calculus of variations (see \cite{Bliss, Bolza, Gelfand, Griffiths, Young}): invoke the method of Lagrange multipliers, integrate by parts, then apply the fundamental lemma of the calculus of variations.  This shows that  $\gamma_*$ satisfies the Euler-Lagrange equations from Section \ref{equations}, which were apparently equivalent to Hamilton's equations.

Recall that, as in Section \ref{equations}, our horizontal constraint is precisely the zero set of the function
\[G(t,  q, {\dot q})=\tfrac{1}{2}x\dot y - \tfrac{1}{2}y \dot x-\dot z\]
and our modified action functional is 
\[ A_{\lambda}(\gamma) = \int_0^T  L_{\lambda}(t, \gamma, \dot \gamma) dt,\]
where $\lambda=\lambda(t)$ is a scalar and we have written $L_{\lambda}(t, q, \dot q)= L(t, q, \dot q)-\lambda(t)G(t, q, \dot q)$.

\begin{lemma}[Lagrange multipliers]\label{multipliers}
If $c$ is a critical point of the action $A$ restricted to horizontal curves, then there exists some $\lambda=\lambda(t)$ such that $c$ is a critical point of $A_{\lambda}$.
\end{lemma}

\begin{proof}
This is a classical result whose various proofs may be found in, for example, \cite{Bliss}, Section 39 of \cite{Bolza}, Section 12 of \cite{Gelfand}, Section IV(e) of \cite{Griffiths}, or Volume II of \cite{Young}.   

\end{proof}

\begin{prop}
Our $\gamma_*$ satisfies the equation
\[\frac{\partial L_{\lambda}}{\partial \gamma_*}- \frac{d}{dt}\bigg( \frac{\partial L_{\lambda}}{\partial \dot \gamma_*}  \bigg)=0\]
for some $\lambda$.
\end{prop}

\begin{proof}
According to Lemma \ref{multipliers}, $\gamma_*$  is a critical point of $A_{\lambda}$ for some $\lambda$ which we now fix.
A standard calculation and the fundamental lemma of the calculus of variations then give
\[0=\int_0^T\left( \frac{\partial L_{\lambda}}{\partial \gamma_*}\eta(t)
-\frac{d}{dt}\frac{\partial L_{\lambda}}{\partial \dot \gamma_*} \eta(t) \right) dt 
+\left(\frac{\partial L_{\lambda}}{\partial \dot \gamma_*}(T)-\frac{\partial L_{\lambda}}{\partial \dot \gamma_*}(0)\right)\eta(0),
\]
for any $\eta$ which is twice differentiable and satisfies the periodicity condition $\eta(0)=\eta(T)$.
Choosing suitable test functions $\eta$, we must have
\begin{equation}\label{EL2}
\frac{\partial L_{\lambda}}{\partial \gamma_*}- \frac{d}{dt}\bigg( \frac{\partial L_{\lambda}}{\partial \dot \gamma_*}  \bigg)=0
\end{equation}
and
\begin{equation}\label{themomentathemselvesarenecessarilyperiodic}
\frac{\partial L_{\lambda}}{\partial \dot \gamma_*}(T)=\frac{\partial L_{\lambda}}{\partial \dot \gamma_*}(0).
\end{equation}

\end{proof}

Then  \eqref{EL2} yields the version of the Euler-Lagrange equations given in \eqref{EL} which were seen to be equivalent to Hamilton's equations.  Thus, $\gamma_*$ satisfies the equations of motion.
Also, \eqref{themomentathemselvesarenecessarilyperiodic} simplifies to the three equations
\begin{align*}
p_x(0)&=p_x(T) \\
p_y(0)&=p_y(T) \\
p_z(0)&=p_z(T),
\end{align*}
where $\gamma_*(t)=(x(t),y(t),z(t),p_x(t), p_y(t), p_z(t))$.
This guarantees that our curve $\gamma_*$ is periodic in all of phase space, not just in configuration space.

The proof of Theorem \ref{periodic} is now complete.

\bibliographystyle{amsplain}

\end{document}